\documentclass[10pt,a4paper,final,oneside,openright,reqno]{amsart} 

 \usepackage{latexsym, amsfonts, amsthm, amsmath, amssymb, tikz, enumitem, booktabs, float,color,multicol,epsf, url,epsfig,epstopdf, array, setspace, graphicx, inputenc, csquotes, xcolor,tikz-cd}
\usepackage{amsfonts}
\usepackage{amsmath}
\usepackage{amssymb}
\usepackage{url}
\usepackage{hyperref}
\usepackage{color}
\usepackage{lipsum}
\usepackage{mathtools}
 \usepackage[T1]{fontenc}
\usepackage{tikz-cd}
 \usepackage{url} 
 \usepackage{hyperref} 
\usepackage{pdfpages}
\usepackage{graphicx}

\usepackage[margin=0.915in,footskip=0.25in]{geometry}

\newtheorem{theorem}{{Theorem}}[section]

\newtheorem{proposition}[theorem]{{Proposition}}

\newtheorem{lemma}[theorem]{{Lemma}}
\newtheorem{corollary}[theorem]{{Corollary}}

\newtheorem{remark}[theorem]{{Remark}}

\newtheorem{example}[theorem]{{Example}}

\definecolor{greenbf}{rgb}{0, 0.7 ,0.3}
 \hypersetup{
colorlinks=true, breaklinks=true, urlcolor= blue, linkcolor= blue,
citecolor= {blue}, }

\def\C{\mathbb{C}}
\def\R{\mathbb{R}}
\def\Q{\mathbb{Q}}

\def\Z{\mathbb{Z}}
\def\S{\mathbb{S}}

\def\T{\mathbb{T}}
\def\D{\mathbb{D}}

\def\HH{\mathcal{H}}
\def\LL{\mathcal{L}}
\def\J{\mathcal{J}}
\def\I{\mathcal{I}}

\def\Hol{{\operatorname{Hol}}}
\def\Loop{{\operatorname{Loop}}}

\def\Hom{{\operatorname{Hom}}}

\def\Flux{{\operatorname{Flux}}}
\def\Ham{{\operatorname{Ham}}}

\def\SO{{\operatorname{SO}}}

\def\U{{\operatorname{U}}} 

\def\Sp{{\operatorname{Sp}}}

\def\Sp{{\operatorname{Sp}}}

\def\SCont{{\operatorname{SCont}}}
\def\Cont{{\operatorname{Cont}}}
\def\Symp{{\operatorname{Symp}}}

\def\CSp{{\operatorname{CSp}}}
\def\Index{{\operatorname{Index}}}

\def\Diff{{\operatorname{Diff}}}
\def\Id{{\operatorname{Id}}}

\author{Souheib Allout}
\address{Fakult\"{a}t f\"{u}r Mathematik,
Ruhr-Universit\"at Bochum, Bochum, Germany}
\email{souheib.allout@rub.de}
\author{Murat Sa\u glam}
\address{Mathematisches Institut,
Universit\"at zu K\"oln, K\"oln, Germany}
\email{msaglam@math.uni-koeln.de}

\begin{document}

\title[]{On contact mapping classes of prequantizations}

\begin{abstract}
We present examples of  prequantizations over integral symplectic manifolds which admit infinitely many smoothly trivial contact mapping classes. These classes are given by the connected components of the strict contactomorphism group which project to the identity component of the symplectomorphism group of the base manifold. Along the way, we study the lifting problem of symplectomorphisms of the base manifold to strict contactomorphisms of the prequantization.
\end{abstract}

\maketitle
 
\section{Introduction} One of the fundamental questions in contact topology is to understand the group $\Cont(V,\xi)$ of co-orientation preserving contactomorphisms of a given co-oriented  contact structure $\xi$ on a closed manifold $V$. Understanding $\Cont(V,\xi)$ requires studying its identity component $\Cont^0(V,\xi)$ and the contact mapping class group $\pi_0(\Cont(V,\xi))$. The aim of this paper is to study  $\pi_0(\Cont(V,\xi))$ for a particular class of contact manifolds called prequantizations. These are total spaces of principal circle bundles $V\to M$, with a connection 1-form  $\lambda$ which is a contact form whose Reeb vector field generates the circle action and, therefore, the curvature form of $\lambda$ is an integral symplectic form $\omega$ on $M$. 

Concerning the contact mapping class group, it is interesting to study the homomorphism
\begin{equation*}\label{J}
\J:\pi_0(\Cont(V,\xi))\to \pi_0(\Diff(V))    
\end{equation*} 
where  $\Diff(V)$ is the group of diffeomorphisms of $V$ which preserve the orientation determined by the co-orientation of $\xi$. Our main results, which rely solely on algebraic topological methods, provide a large class of examples of prequantizations $(V,\lambda)$ for which the kernel of $\J$ contains an infinite subgroup of $\pi_0(\Cont(V,\xi))$. This subgroup is given as follows. Given a prequantization $(V,\lambda)\to (M,\omega)$, we get a natural homomorphism 
\begin{equation}\label{scont to symp at intro} 
\pi_0(\SCont(V,\lambda))\to \pi_0(\Symp(M,\omega))    
\end{equation}
where $\SCont(V,\lambda)$ is the group of  strict contactomorphisms, that is,  contactomorphisms which preserve the contact form $\lambda$, and $\Symp(M,\omega)$ is the group of symplectomorphisms of $(M,\omega)$. It turns out that the kernel of the above homomorphism is isomorphic to $H^1_{dR}(M,\Z)/\Gamma_\omega$, where $\Gamma_\omega$ is the flux group of $(M,\omega)$, see Theorem \ref{comp. of H}. Then we show that for a large class of integral symplectic manifolds for which $\Gamma_\omega$ turns out to be trivial, the composition   
\begin{equation*}\label{I}
\I:H^1_{dR}(M,\Z)\to \pi_0(\SCont(V,\lambda))\to \pi_0(\Cont(V,\xi)) \end{equation*} 
is injective. More precisely we prove the following. 
\begin{theorem}\label{main thm at intro}
Let $(M,\omega)$ be a closed integral symplectic manifold such that either $[\omega]$ vanishes on $\pi_2(M)$ or $(M,\omega)$ is monotone on $\pi_2(M)$ with a non-vanishing monotonicity constant. Suppose, in addition, that any abelian subgroup of $\pi_1(M)$ has rank at most one. If $(V,\lambda)\to (M,\omega)$ is a prequantization, then the homomorphism $\I$ is injective.
\end{theorem}
The idea of the proof is as follows. For a given class $\alpha \in H_{dR}^1(M,\Z)$, the associated strict contactomorphism $L_\alpha$ acts on a loop $\delta$ in the conformally symplectic frame bundle $\CSp(\xi)$, up to free homotopy, by concatenation with an iterate of the \emph{Reeb frame loop}, which is defined by picking a symplectic frame of a contact hyperplane and pushing it forward by the Reeb flow along a circle fiber. The amount  of iteration of the Reeb frame loop is $\langle [\alpha], [\overline{\delta}]\rangle$, where $\overline{\delta}$ is the projection of $\delta$ to $M$. Then, we associate a Maslov-type index to the Reeb frame loop which we use to show that the resulting loop $L_\alpha\circ \delta$ is not freely homotopic to $\delta$ whenever $\langle [\alpha], [\overline{\delta}]\rangle\neq 0$. Using this Maslov-type index, we show the following statement
\begin{proposition}\label{reeb loop frame inf order at intro} Let $(V,\lambda)\to (M,\omega)$ be a prequantization. Suppose that either $[\omega]$ vanishes on $\pi_2(M)$ or $(M,\omega)$ is monotone on $\pi_2(M)$ with non-vanishing monotonicity constant. Then the Reeb frame loop is of infinite order in $\pi_1(\CSp(\xi))$. 
 In particular, the Reeb flow, seen as a loop in $\Cont(V,\xi)$, is of infinite order in $\pi_1(\Cont(V,\xi))$.  
\end{proposition}
We note that the Maslov index that we associate to the Reeb frame loop can be seen as a reformulation of the first Maslov index, which is a homomorphism from $\pi_1(\Cont(V,\xi))$ to $\Z$, of Casals, Ginzburg and Presas \cite{CGP}, applied to the Reeb loop (see also \cite{CP} and \cite{CS}). A similar result was also obtained by Albers, Shelukhin, and Zapolsky, as an application (among other applications) of some spectral invariants which they associate to contactomorphisms of prequantization bundles (see \cite{ASZ} for more details).

\begin{theorem}\label{cor main at intro} Let $(\Sigma_g,\omega_0)$ be a surface of genus $g\geq 2$ with an integral symplectic form and $\omega_{FS}$ be the Fubini-Study form on $\C P^n$, which is normalized so that the standard $\C P^1\subset \C P^n$ has unit area. Consider a prequantization  
$(V,\lambda)\to (\C P^n\times \Sigma_g, \omega_{FS}\oplus \omega_0).$ Then the homomorphism 
$$\mathcal{I}:H^1_{dR}(\Sigma_g,\Z)\to \pi_0(\Cont(V,\xi)),$$ 
is injective. Moreover, $\ker(\J)$  contains the image of $\I$ for $n$ odd, and contains $\I(2H^1_{dR}(\Sigma_g,\Z))$ for $n$ even.
\end{theorem}
More generally, we describe a setting for which the composition $\J\circ\I$ fails to be injective, see Proposition \ref{smooth}. In principle, this provides examples other than the ones given in the above theorem.

In what follows we give an account of previous results in the direction of this work. The first result we want to mention is given by Giroux \cite{Gir}, which says that $\J$ is injective for $V=\T^3$ and $\xi$ the standard tight contact structure (See also \cite{GK} Proposition 2.). The first example, according to our knowledge, of a closed contact manifold for which the kernel of $\J$ is non-trivial is attributed to Chekanov, see \cite{EF, Vog}. The statement here is that for any overtwisted contact structure $\xi$ on $\S^3$, $\pi_0(\Cont(\S^3,\xi))$ is either $\Z_2$ or $\Z_2\oplus \Z_2$ depending on the Hopf invariant. Later in \cite{Dy}, Dymara introduces an invariant for overtwisted contact structures on $\S^3$, which enables her to show that a homomorphism from $\pi_0(\Cont(\S^3,\xi))$ to $\Z_2$ that she constructs is surjective for every overtwisted $\xi$.

The first example where the contact structure is tight and also the kernel of $\J$ is infinite is given by Ding and Geiges in \cite{DG}, following an observation made by Gompf in \cite{Gom}. Here the contact manifold is $\S^1\times \S^2$ with the standard tight contact structure. Later in \cite{GM}, Giroux and Massot show that the kernel of $\J$ is $\Z_d$ for $V$ being the $d-$fold fiber cover of the unit cotangent bundle of a suitable surface of genus at least $2$. In \cite{MK}, Massot and Niederkr\"uger construct hypertight contact manifolds in every odd dimension for which the kernel of $\J$ contains non-trivial but finite order elements. Later in \cite{LZ} Lanzat and Zapolsky construct an embedding of certain (Braid) subgroups $G$ of the symplectic mapping class group of the $A_m-$Milnor fiber, which is non-compact and exact, to the contact mapping class group of the associated contactization and show that $G$ is contained in the kernel of $\J$. 

The last two works we want to mention are due to Gironella. In \cite{Gi} he constructs examples of overtwisted contact manifolds, in every odd dimension, admitting smoothly trivial finite order contact mapping classes. Lastly, Gironella in \cite{Gi2}, generalizes the constructions in \cite{DG} to provide examples of tight contact manifolds in all odd dimensions, which admit smoothly trivial contact mapping classes of infinite order. 

\subsection*{Questions and further discussion}\label{sec4} Given a prequantization bundle $(V,\lambda)$ over an integral symplectic manifold $(M,\omega)$, the kernel of the homomorphism \eqref{scont to symp at intro} is naturally identified with $H^1_{dR}(M,\Z)/\Gamma_\omega$, see Theorem \ref{comp. of H}, and that one obtains a homomorphism  $$\I: H^1_{dR}(M,\Z)/\Gamma_\omega\to \pi_0(\Cont(V,\xi)).$$
We do not know any example for which the homomorphism $\I$ fails to be injective. 

One may perhaps find situations where the image of $\I$ contains elements which are  isotopic to the identity in the group of formal contactomorphisms, namely the group of pairs $(\varphi,\Phi)$, where $\varphi\in \Diff(V)$ and $\Phi$ is a bundle map of $TV$ covering $\varphi$ and preserving $\xi$ conformally symplectically. A candidate setting for this is when the Reeb frame loop is contractible, in other words, when $M$ admits an element in $\pi_2(M)$ of unit symplectic area on which $c_1(TM)$ vanishes.

On the other hand, the non-injectivity of $\I$ has the potential to provide examples of strict contactomorphisms that are contact isotopic to the identity but without translated points in the sense of Sandon, see \cite{San}. The setting of such an example could be that there is some $\alpha\in \ker \I$ represented by a non-singular closed $1-$form on $M$ such that the time-one-map of its symplectic gradient flow is without fixed points. Recently, Cant \cite {Can} provided examples of contactomorphisms on standard contact spheres that are contact isotopic to identity yet without translated points. However there is no known example of a strict contactomorphism with these properties. 
\subsection*{Outline of the paper} In Section \ref{sec2}, we discuss generalities about prequantizations, and study the homomorphism \eqref{scont to symp at intro}. In Section \ref{sec3}, we give our main results, namely Theorem \ref{main thm}, Proposition \ref{smooth} and Corollary \ref{cor main}. 
\subsection*{Acknowledgements} We thank Felix Schlenk for initiating this work, for many useful discussions and comments on this manuscript. We thank Alberto Abbondandolo, Stefan Nemirovski and Stefan Suhr for very helpful discussions and comments on this manuscript and also for providing encouragement. We also thank Hansj\" org Geiges for his comments during the presentation of this work in a series of lectures in the Arbeitsgemeinschaft Symplektische Topologie at University of Cologne. We thank the referee for the comments. 

The first author is fully and the second author is partially supported  by the SFB/TRR 191 `Symplectic Structures in Geometry, Algebra and Dynamics', funded by the DFG (Projektnummer 281071066 - TRR 191).
\section{Generalities on strict contactomorphisms of prequantizations}\label{sec2}
In this section we fix our setting, recall some 
generalities on prequantizations and determine the kernel of the homomorphism \eqref{scont to symp at intro}. Most of the results presented in this section can be found in the classical books of Banyaga \cite{Ban} and Souriau \cite{So}. We try to keep our presentation minimal and we refer to \cite{Ban, So} for further details.

Let $(M,\omega)$ be a $2n$-dimensional closed integral symplectic manifold, that is, the value of $[\omega]$ over any integral homology 2-cycle is an integer. Then a classical fact due to Kobayashi provides our main object of study. We fix $\S^1:=\R/\Z$ for the rest of this manuscript.
\begin{theorem}[\cite{Ko, BoW}]\label{thm preq} Let $\pi:V\to M$ be a principal $\S^1-$bundle whose Euler class is a lift of $-[\omega]$. Then there is a contact form $\lambda$ on the total space $V$ such that
\begin{enumerate} [label=\text{(p\arabic*)}]
    \item \label{reeb} The vector field that generates the $\S^1-$action is the Reeb vector field of $\lambda$.
    \item \label{curvature} $\pi^*\omega=d\lambda$.
\end{enumerate}
\end{theorem} 

We call such a bundle $(V,\lambda)$, together with a contact form as above, a \textit{prequantization} over $(M,\omega)$.

Given a prequantization $(V,\lambda)$ over $(M,\omega)$ we want to understand the induced homomorphism 
\begin{equation}\label{Scont to Symp}
 \SCont(V,\lambda)\to \Symp(M,\omega),\;\;\varphi\mapsto \overline{\varphi}   
\end{equation}
where $\SCont(V,\lambda)$ is the group of strict contactomorphisms of $(V,\lambda)$, that is, diffeomorphisms of $V$ that preserve $\lambda$, and $\Symp(M,\omega)$ is the group of symplectomorphisms of $(M,\omega)$. Notice that any strict contactomorphism $\varphi$ of $(V,\lambda)$ is in particular an $\S^1-$equivariant diffeomorphism of $V$ and therefore descends to a diffeomorphism $\overline{\varphi}$ of $M$, which preserves $\omega$ due to \ref{curvature}.
 
Our main objective is to characterize the image and the kernel of the map \eqref{Scont to Symp} and of the induced map
\begin{equation}\label{Scont to Symp on components}
 \pi_0(\SCont(V,\lambda))\to \pi_0(\Symp(M,\omega))   
\end{equation}
on mapping class groups. The surjectivity of the map \eqref{Scont to Symp}, respectively of the map \eqref{Scont to Symp on components}, relies on the answer to the question whether a given $\overline{\varphi}\in \Symp(M,\omega)$ can be lifted (maybe after an isotopy) to a map $\varphi\in\SCont(V,\lambda)$. 

In order to study this lifting problem, we fix the $\S^1-$bundle $\pi:V\to M$ as in Theorem \ref{thm preq} and consider the space $\mathcal{P}$ of all contact forms on $V$ that satisfy \ref{reeb} and \ref{curvature}. Notice that $\mathcal{P}$ is an affine space given by
\begin{equation}\label{P affine}
\mathcal{P}=\Bigl\{\lambda+\beta\,|\,\beta\in\Omega^1_{cl}(V)\;{\rm and}\; \beta(X)=0\Bigr\}=\Bigl\{\lambda+\pi^*\alpha\,|\,\alpha\in\Omega^1_{cl}(M)\Bigr\}
\end{equation}
for any fixed $\lambda\in \mathcal{P}$, where $\Omega^*_{cl}$ denotes the space of closed forms. Let $\Diff_{\S^1}(V)$ be  the group of $\S^1-$equivariant diffeomorphisms of $V$. We consider the action of the subgroups  
\begin{equation}\label{G^omega}
    \Diff_{\S^1}^{\omega}(V)=\bigl\{\varphi\in\Diff_{\S^1}(V)\,|\,\overline{\varphi}^*\omega=\omega\bigr\}
\end{equation}
and
\begin{equation}\label{GId}
    \Diff_{\S^1}^{\Id}(V)=\bigl\{\varphi\in\Diff_{\S^1}(V)\,|\,\overline{\varphi}=\Id\bigr\}.
\end{equation}
on $\mathcal{P}$ by the pull-back.
We note that $\varphi\in \Diff_{\S^1}^{\omega}(V)$ if and only if $\varphi\in \Diff_{\S^1}(V)$ and $\varphi^*(d\lambda)=d\lambda$ for all $\lambda\in \mathcal{P}$ and the stabilizer of any $\lambda\in \mathcal{P}$ in $\Diff_{\S^1}^{\omega}(V)$ is precisely $\SCont(V,\lambda)$. On the other hand we have the identification 
\begin{equation}\label{GId=maps}
\Diff_{\S^1}^{\Id}(V)\cong C^\infty(M,\S^1)      
\end{equation}
since any $\varphi\in\Diff_{\S^1}^{\Id}(V)$ is given by 
\begin{equation}\label{GId as map}
\varphi(p)=\phi^X_{f(\pi(p))}(p)    
\end{equation}
for a unique smooth map $f:M\to \S^1$, where $\phi^X_t$ denotes the flow of the vector field $X$ that generates the $\S^1-$action.

\begin{theorem} \label{surjective on comp}Let $(V,\lambda)$ be a prequantization over $(M,\omega)$ and assume that $H^2(M,\Z)$ is torsion-free. Then the homomorphism \eqref{Scont to Symp on components} is surjective.     
\end{theorem}
\begin{proof} We first show that 
the identity component of $\Diff_{\S^1}^{\omega}(V)$ acts transitively on $\mathcal{P}$. In fact, given the line segment $\lambda_t=\lambda+t\beta$, $t\in [0,1]$ joining $\lambda$ and $\lambda+\beta$,  we define the vector field $Y$ on $V$ by the equations $\lambda(Y)=0$ and $\beta=-\imath_Y\,d\lambda$. Note that $\dot{\lambda}_t=\beta$, $d\lambda_t=d\lambda$ and $\lambda_t(Y)=0$ for all $t\in[0,1]$. Now let $\phi_t$ be the flow of $Y$. We get
$$\frac{d}{dt}\phi_t^*\lambda_t=\phi_t^*\left(\dot{\lambda_t}+\mathcal{L}_Y\lambda_{t}\right)=\phi_t^*\left(\beta+\imath_Y d\lambda_{t}+d(\lambda_{t}(Y))\right)=
      \phi_t^*\left(\beta+\imath_Y d\lambda\right)=0.
$$
Hence $\phi_t^*(\lambda+t\beta)=\lambda$ and $\phi_t\in \Diff_{\S^1}^{\omega}(V)$ for all $t\in[0,1]$. In fact $[X,Y]=0$ since 
$$0=\lambda([X,Y])=X(\lambda(Y))-Y(\lambda(X))-d\lambda(X,Y)$$
and 
$$0=\imath_{[X,Y]}d\lambda=\mathcal{L}_X\imath_Yd\lambda-\imath_Y\mathcal{L}_Xd\lambda=-\mathcal{L}_X\beta=0.$$
Note that any $\varphi\in \Symp(M,\omega)$ lifts to some $\widetilde{\varphi}\in \Diff_{\S^1}^\omega(V)$ since $\varphi$ pulls back the circle bundle $V\to M$ to a circle bundle on $M$, which has the same Euler class since $H^2(M,\Z)$ is torsion-free. The resulting bundle equivalence provides $\widetilde{\varphi}$. Then there is some  $\varphi$ in the identity component of $\Diff_{\S^1}^{\omega}(V)$ such that $\psi^*(\widetilde{\varphi}^*\lambda)=\lambda$. So $\widetilde{\varphi}\circ \psi\in \SCont(V,\lambda)$ and $\overline{\widetilde{\varphi}\circ\psi}$ is isotopic to $\varphi$ in $\Symp(M,\omega)$.
\end{proof}

\begin{lemma}\label{kernel} Let $\lambda\in \mathcal{P}$ and $\varphi\in\SCont(M,\lambda)$ such that $\overline{\varphi}=\Id$. Then $\varphi=\phi_t^X$ for some $t\in \S^1$.    
\end{lemma}
\begin{proof} Given $\varphi\in \Diff_{\S^1}^{\Id}(V)$ as in \eqref{GId as map}, an easy computation shows that 
\begin{equation}\label{GId on P}
\varphi^*\lambda=\lambda+\pi^*df    
\end{equation}
where $df$ is the differential of $f:M\to\S^1$ viewed as the 1-form $f^*dt$. Now if $\varphi$ preserves $\lambda$, then $f$ is a constant map and therefore $\varphi=\phi^X_t$ for some $t\in\S^1$.
\end{proof}
\begin{theorem}\label{surjective}Let $(V,\lambda)$ be a prequantization over $(M,\omega)$ and assume that $H^2(M,\Z)$ is torsion-free. Then the following are equivalent. 
\begin{enumerate}
    \item The action of $\Diff_{\S^1}^{\Id}(V)$ on $\mathcal{P}$ is transitive.
    \item The first cohomology group $H_{dR}^1(M,\R)$ is trivial.
    \item The homomorphism \eqref{Scont to Symp} is surjective. In other words, the sequence 
    $$0\to\S^1\to \SCont(V,\lambda)\to \Symp(M,\omega)\to 0$$
    is exact.
\end{enumerate}    
\end{theorem}
\begin{proof} We recall that every class in $H_{dR}^1(M,\Z)$ is represented by the differential of a circle valued function. Combining this fact with \eqref{GId on P} gives the non-canonical identification 
$$\mathcal{P}/\Diff_{\S^1}^{\Id}(V)\simeq H_{dR}^1(M,\R)/H_{dR}^1(M,\Z),$$
which shows the equivalence of (1) and (2). 

Now assume that $\Diff_{\S^1}^{\Id}(V)$ acts transitively on $\mathcal{P}$ and let $\varphi\in  \Symp(M,\omega)$. Then $\varphi$ lifts to some $\widetilde{\varphi}\in \Diff_{\S^1}^\omega(V)$ as in the proof of Theorem \ref{surjective on comp}. We have some $\psi\in \Diff_{\S^1}^{\Id}(V)$ such that $\psi^*(\widetilde{\varphi}^*\lambda)=\lambda$. Hence $\widetilde{\varphi}\circ\psi$ is in $\SCont (V,\lambda)$ and lifts $\varphi$. Together with Lemma \ref{kernel} we get the exact sequence above.
For the converse statement, we take $\lambda'\in \mathcal{P}$ and by the proof of Theorem \ref{surjective on comp}, there is some $\varphi$ in the identity component of $\Diff_{\S^1}^{\omega}(V)$ such that $\varphi^*\lambda'=\lambda$ and $\overline{\varphi}\in \Symp(M,\omega)$. By assumption there is a lift $\widetilde{\varphi}\in\SCont(V,\lambda)$ of $\overline{\varphi}$. Then we get $\varphi\circ\widetilde{\varphi}^{-1}\in \Diff_{\S^1}^{\Id}(V)$ and $(\varphi\circ\widetilde{\varphi}^{-1})^*\lambda'=\lambda$. 
\end{proof}
\begin{remark} The assumption that $H^2(M,\Z)$ is torsion-free is essential for  Theorems \ref{surjective on comp} and  \ref{surjective}. It is possible to construct an integral symplectic manifold $(M,\omega)$ with non-trivial torsion in $H^2(M,\Z)$ and a symplectomorphism that does not preserve a class $e\in H^2(M,\Z)$ that lifts $-[\omega]$. Such a symplectomorphism can not be lifted even to an $\S^1$-equivariant map of the circle bundle with Euler class $e$. 
\end{remark}

Given a prequantization  $(V,\lambda)\to (M,\omega)$, we now want to understand the kernel of the map \eqref{Scont to Symp on components}. To this end we define the following subgroups 
\begin{equation}\label{L}
    \LL:=\bigl\{\varphi\in\Symp(M,\omega)\,|\,\varphi \textrm{ lifts to }\SCont(V,\lambda)\bigr\}
\end{equation}
and 
\begin{equation}\label{H}
    \HH:=\LL\cap\Symp^0(M,\omega)
\end{equation}
where $\Symp^0(M,\omega)$ is the identity component. 

By Lemma \ref{kernel}, the map \eqref{Scont to Symp} now fits into the exact sequence 
\begin{equation}\label{exact seq for scont}
    0\to\S^1\to \SCont(V,\lambda)\to\LL\to 0.
\end{equation}
which defines a Serre fibration. In fact, it is easy to see that once restricted to the identity components, the above sequence reads 
$$0\to\S^1\to \SCont^0(V,\lambda)\to\Ham(M,\omega)\to 0$$
where $\Ham(M,\omega)$ is the group of Hamiltonian diffeomorphisms. Then the homotopy lifting property for $0-$cells follows from the non-trivial fact that any path of Hamiltonian diffeomorphism starting at $\Id$ is a Hamiltonian isotopy. For the details, we refer to \cite{BW}.     

An immediate consequence of the above lemma is the following exact sequence of mapping class groups. 
\begin{equation}\label{comp of scont}
    0\to \pi_0(\SCont(V,\lambda))\to\pi_0(\LL)\to 0.
\end{equation}
On the other hand if $H^2(M,\Z)$ is torsion-free we have the exact sequence
\begin{equation}
    0\to \pi_0(\HH)\to\pi_0(\LL)\to\pi_0(\Symp(M,\omega))\to 0.
\end{equation}

It turns out that the elements of $\LL$ are characterized by their effect on holonomy and this effect is explicitly determined for the elements of $\HH$. Let $\gamma:[0,1]\to M$ be a piece-wise smooth curve and $p\in \pi^{-1}(\gamma(0))$, then there exists a unique curve $\widehat{\gamma}:[0,1]\to V$ such that  $$\widehat{\gamma}(0)=p, \ \ \pi\circ \widehat{\gamma}=\gamma\ \ \text{and}\ \ \lambda(\dot{\widehat{\gamma}})=0.$$
We call $\widehat{\gamma}$, suppressing its initial condition, \textit{the horizontal lift} of $\gamma$. If $\gamma$ is a loop, that is $\gamma(0)=\gamma(1)$, then the \textit{holonomy} $\Hol(\gamma)\in\S^1$ of $\gamma$ is defined by 
\begin{equation} \phi^X_{\Hol(\gamma)}(\widehat{\gamma}(1))=\widehat{\gamma}(0).    
\end{equation}
We note that the holonomy does not depend on $\widetilde{\gamma}(0)$ and  it is invariant under positive reparameterizations of $\gamma$. Moreover,  $\Hol(\gamma^{-1})=-\Hol(\gamma)$ where $\gamma^{-1}$ some (hence any) negative reparametrization of $\gamma$.  
Now given a diffeomorphism $\varphi$ of $M$ one considers its effect on holonomy, namely the map 
\begin{equation}\label{holonomy effect}
E_\varphi: \Loop(M)\to \S^1,\ \ \gamma\mapsto \Hol(\varphi\circ\gamma)-\Hol(\gamma) \end{equation}
where $\Loop(M)$ is the space of piece-wise smooth maps from $\S^1$ to $M$. The following fact, which we spell out in accordance with our purposes, is due to Kostant  and says that the holonomy effect determines whether a symplectomorphism $\varphi$ lifts to $\SCont(V,\lambda)$. 
\begin{theorem}[Kostant, \cite{Kos} Proposition 3.3]\label{kostant} A symplectomorphism $\varphi\in \Symp(M,\omega)$ admits a lift to $\SCont(V,\lambda)$ if and only if $E_{\varphi}=0$. 
\end{theorem} 
\begin{proof} We fix a point $x\in M$ and put $y=\varphi(x)$. After trivializing the circle bundle near $x$ and $y$, we fix a map $\eta:\pi^{-1}(x)\to \pi^{-1}(y)$, which reads as a rotation with respect to any other choice of such trivializations. Now given $p\in V$  we take a path $\gamma$ such that $\gamma(0)=\pi(p)$ and $\gamma(1)=x$ and we define $\widetilde{\varphi}(p)$ by the equation 
$$\widehat{\varphi\circ \gamma}(1)=\eta(\widehat{\gamma}(1))$$
where $\widehat{\varphi\circ \gamma}$ is the horizontal lift of $\varphi\circ \gamma$ with $\widehat{\varphi\circ \gamma}(0)=\widetilde{\varphi}(p)$ and $\widehat{\gamma}$ is the horizontal lift of $\gamma$ with $\widehat{\gamma}(0)=p$. 

Note that $\widetilde{\varphi}$ is well-defined if and only if $E_\varphi=0$. Moreover $\widetilde{\varphi}$, whenever is well-defined, is an $\S^1-$equivariant lift of  $\varphi$, which maps horizontal lifts of curves in $M$ to horizontal lifts of their images under $\varphi$. It is not difficult to show that this property together with the fact that $\varphi$ is symplectic implies that $\widetilde{\varphi}$ is in $\SCont(V,\lambda)$.    
\end{proof}
Given $\varphi\in \Symp^0(M,\omega)$, it is easy to understand its effect on  holonomy. We begin with the following observation. 
\begin{lemma}\label{hol}
    Let $\gamma_1, \gamma_2, \ldots, \gamma_k:\S^1\to M$ be a family of smooth loops. Suppose there is an oriented surface $\Sigma$ with boundary $\partial \Sigma=S_1 \cup S_2 \cup \ldots \cup S_k$ consisting of $k$ circles and a smooth map $\sigma:\Sigma\to M$ extending each $\gamma_j$, that is, for each $j$, once $S_j$ is oriented as the boundary of $\Sigma$ and identified with $\S^1$,   $\sigma|_{S_j}=\gamma_j$. Then we have  $$\sum_{j=1}^{k}\Hol(\gamma_j)=\int_{\Sigma}\sigma^{*}\omega \mod \Z$$
\end{lemma}
\begin{proof} We first note that for any smooth $\gamma:\S^1\to M$ 
$$\Hol(\gamma)=\int_{\widetilde{\gamma}}\lambda \mod \Z$$
where $\widetilde{\gamma}$ is any loop that lifts $\gamma$. In fact given a lift $\widetilde{\gamma}:\S^1\to V$, we consider the horizontal lift $\widehat{\gamma}$ of $\gamma$ with $\widehat{\gamma}(0)=\widetilde{\gamma}(0)$ and define a smooth
map $\tau:[0,1]\to \S^1$ such that  
$$\phi^X_{\tau(s)}(\widehat{\gamma}(s))=\widetilde{\gamma}(s).$$
Note that $\tau(1)=\Hol(\gamma)$ modulo $\Z$. We lift $\tau$ to a map into $\R$, still denoted by $\tau$, such that $\tau(0)=0$ and consider the smooth map $$\eta:\Delta\to V,\ (s,t)\mapsto \phi^X_{t}(\widehat{\gamma}(s))$$ 
on the domain
$\Delta:=\bigl\{(s,t)\in \R^2\ | 0\leq t\leq \tau(s),\ s\in [0,1]\bigr\}$.
Note that $\eta^*d\lambda=0$. By Stokes' theorem
$$0=\int_\Delta \eta^*d\lambda= \int_{\widehat{\gamma}}\lambda+\tau(1)-\int_{\widetilde{\gamma}}\lambda$$
and the claim follows.

Now let $\sigma:\Sigma\to M$ be as in the statement. Since $\Sigma$ has the homotopy type of a graph, there exists a lift $\widetilde{\sigma}:\Sigma\to V$ of $\varphi$. By the above observation we get
$$\int_{\Sigma}\sigma^{*}\omega=\int_{\Sigma}\widetilde{\sigma}^{*}d\lambda=\int_{\partial \Sigma}\widetilde{\sigma}^*\lambda=\sum_{j=1}^{k}\Hol(\gamma_j) \mod \Z.$$   
\end{proof}
An immediate consequence of this lemma is the following.
\begin{proposition} For any $\varphi\in\Symp(M,\omega)$, the map \eqref{holonomy effect}
descends to a homomorphism 
$$E_\varphi:H_1(M,\Z)\to \S^1,\ \ [\gamma]\mapsto E_\varphi(\gamma).$$
Moreover the map
\begin{equation}\label{holonomy homo}
\Symp^0(M,\omega)\to \Hom(H_1(M,\Z),\S^1),\ \ \varphi\mapsto E_\varphi.    
\end{equation}
is a homomorphism. 
\end{proposition}
\begin{proof} Given smooth loops $\gamma_1,\gamma_2$ representing the same class in $H_1(M,\Z)$, there is an oriented surface $\Sigma$ with two oriented boundary components $S_1$ and $S_2$ and a smooth map $\sigma:\Sigma\to M$ such that $\sigma|_{S_1}=\gamma_1$ and $\sigma|_{S_2}=\gamma_2^{-1}$. We get 
$$\Hol(\gamma_1)+\Hol(\gamma_2^{-1})=\Hol(\gamma_1)-\Hol(\gamma_2)=\int_\Sigma \sigma^*\omega$$
and therefore
$$E_\varphi(\gamma_1)-E_\varphi(\gamma_2)=\int_\Sigma (\varphi\circ\sigma)^*\omega-\int_\Sigma\sigma^*\omega=\int_\Sigma \sigma^*(\varphi^*\omega)-\int_\Sigma\sigma^*\omega=0.$$
Now given $\varphi_1,\varphi_2\in\Symp^0(M,\omega)$ and $[\gamma]\in H_1(M,\Z)$, we have $[\varphi_2\circ \gamma]=[\gamma]$ in $H_1(M,\Z)$ and
$$E_{\varphi_1\circ\varphi_2}(\gamma)=E_{\varphi_1}(\varphi_2\circ \gamma)+E_{\varphi_2}(\gamma)=E_{\varphi_1}(\gamma)+E_{\varphi_2}(\gamma).$$
\end{proof}
Notice that $\HH$ is precisely the kernel of the homomorphism \eqref{holonomy homo}. Our next objective is to relate the homomorphism \eqref{holonomy homo} to the flux homomorphism
\begin{equation}\label{flux}
\Flux:\Symp^0(M,\omega)\to H^1_{dR}(M,\R)/\Gamma_\omega. 
\end{equation}
Recall that if  $(\varphi_t)_{t\in[0,1]}$ is a symplectic isotopy  of a closed symplectic manifold $(M,\omega)$, not necessarily integral, generated by the time dependent symplectic vector field $X_t$, the flux of $(\varphi_t)_{t\in[0,1]}$ is defined by
$$\Flux (\varphi_t)= \int_0^1[\imath_{X_t}\omega]dt\in H^1_{dR}(M,\R).$$
It turns out that $\Flux (\varphi_t)$ stays invariant under any homotopy of the path $(\varphi_t)_{t\in[0,1]}$ in $\Symp^0(M,\omega)$, which fixes the end points $\Id$ and $\varphi_1$ and consequently defines a homomorphism 
$$\Flux: \widetilde{\Symp^0}(M,\omega)\to H^1_{dR}(M,\R)$$
where $\widetilde{\Symp^0}(M,\omega)$ is the universal cover of $\Symp^0(M,\omega)$, seen as the space of homotopy classes, in the above sense, of paths in $\Symp^0(M,\omega)$ starting form the identity. Viewing $\pi_1(\Symp^0(M,\omega),\Id)$ as a subgroup of $\widetilde{\Symp^0}(M,\omega)$, one defines the \textit{flux group}
$$\Gamma_{\omega}:=\Flux(\pi_1(\Symp^0(M,\omega),\Id))\subseteq H^1_{dR}(M,\R)$$
and obtains the homomorphism \eqref{flux}. It turns out that $\Flux(\varphi_t)$ vanishes if and only if $(\varphi_t)$ is homotopic to a Hamiltonian isotopy. Hence the kernel of \eqref{flux} is precisely the group $\Ham(M,\omega)$ of Hamiltonian diffeomorphisms.

The following observation relates $\Flux$ to the homomorphism \eqref{holonomy homo} if $\omega$ is integral.   
\begin{lemma}\label{flux=hol}
    Let $ (\varphi_t)$ be a symplectic isotopy starting at $\Id$. Then for any smooth loop $\gamma:\S^1\to M$ we have $$\int_{\gamma}\Flux(\varphi_t)=E_{\varphi_1}(\gamma) \mod \Z.$$
\end{lemma}
\begin{proof}

Given $\gamma$ and $\varphi_t$ as above, one defines the smooth map $\sigma:[0,1]\times \S^1 \to M$ by $\sigma(s,t)=\varphi_t(\gamma(s))$. An easy computation shows that
\begin{equation}\label{flux formula}
\int_{[\gamma]}\Flux(\varphi_t)=\int_{[0,1]\times \S^1}\sigma^*\omega
\end{equation}
and from Lemma \ref{hol} it follows that $$\int_{[0,1]\times \S^1}\sigma^*\omega=\Hol(\gamma^{-1})+\Hol(\varphi_1\circ\gamma)=E_{\varphi_1}(\gamma)\mod \Z.$$
\end{proof}
We have the following characterization of the mapping class group $\pi_0(\HH)$. 
\begin{theorem}\label{comp. of H} We have the following canonical isomorphisms of groups
\begin{equation}\label{comp of HH}
\pi_0(\HH)=\HH/\Ham(M,\omega)\cong H^1_{dR}(M,\Z)/\Gamma_{\omega}.    
\end{equation}
\end{theorem}
\begin{proof}
The statement follows immediately from Lemma \ref{flux=hol} and the fact that $\Gamma_\omega$ is contained in $H^1_{dR}(M,\Z)$ as $\omega$ is integral.
\end{proof}
\section{$H^1_{dR}(M,\Z)$ in $\pi_0(\Cont(V,\xi))$ and in $\pi_0(\Diff(V))$}\label{sec3}
The aim of this section is to provide a class of examples of prequantizations $(V,\lambda)\to (M,\omega)$ such that $\pi_0(\HH)$ is an infinite subgroup of $\pi_0(\SCont(V,\lambda))$ and 
the homomorphism 
\begin{equation}\label{HH to Cont}
\pi_0(\HH)\to \pi_0(\Cont(V,\xi))    
\end{equation}
is injective, while the homomorphism 
\begin{equation}\label{HH to Diff}
\pi_0(\HH)\to \pi_0(\Diff(V))    
\end{equation}
is trivial. Here $\xi$ stands for the kernel of $ \lambda$, $\Cont(V,\xi)$ is the group of co-orientation (given by $\lambda$) preserving contactomorphisms of $(V,\xi)$ and $\Diff(V)$ is the group of orientation (determined by $\lambda$) preserving diffeomorphisms of $V$.

We begin with an explicit description of the inverse of the isomorphism \eqref{comp of HH}. Let $\alpha$ be a closed integral 1-form on $M$ given by $\alpha=df$ for some $f\in C^\infty(M,\S^1)$. We define the symplectic gradient $X_\alpha$ of $\alpha$ via  $-\alpha=\imath_{X_\alpha}\omega$ and we let $(\overline{\varphi}_t)$ denote its flow. Then the time-one-map $\overline{\varphi}_1$ belongs to $\HH$ with a lift 
\begin{equation}\label{L_alpha} 
L_\alpha:V\to V, \ \ L_\alpha=\phi^X_{f\circ\pi}\circ \varphi_1  
\end{equation}
in $\SCont(V,\lambda)$, where $\varphi_1$ is the time-one-map of the flow $(\varphi_t)$ of the horizontal lift $X_h$ of $X_\alpha$ and the map $\phi^X_{f\circ\pi}$ is given by \eqref{GId as map}. Indeed as we saw in the proof of Theorem \ref{surjective on comp}, 
\begin{equation}\label{pullback by hor. flow}
(\varphi_t)^*\lambda=\lambda-t\,\pi^*\alpha.    
\end{equation}
Together with \eqref{GId on P} we see that $L_\alpha$ preserves $\lambda$. We also note that the maps $\phi^X_{f\circ\pi}$ and $\varphi_t$ commute for all $t\in\R$. Indeed as in the proof of Theorem \ref{surjective on comp}, the vector fields $X$ and $X_h$ commute and the map $f\circ\pi$ is invariant under the flow of $X_h$. We have the following observation on the action of $L_\alpha$ on loops in $V$. 
\begin{lemma}\label{L_alpha on pi_1} Let $\delta:\S^1\to V$ be a loop. Then $L_\alpha\circ \delta$ is freely homotopic to the loop $\delta\ast \gamma_V^k$ where 
$\gamma_V:\S^1\to V$ denotes the parameterized Reeb orbit starting at the point $\delta(0)$ and $k=\int_\delta \pi^*\alpha$.    
\end{lemma}
\begin{proof} Note that the loop $L_\alpha\circ \delta$ is freely homotopic to the loop $\phi^X_{f\circ\pi}\circ \delta$ via the isotopy $(\varphi_t)$ so it is enough to show  $\delta\ast \gamma_V^k$ is freely homotopic to the loop $\phi^X_{f\circ\pi}\circ \delta$. To this end we consider  the pull-back bundle $\overline{\delta}^*V\to \S^1$, where $\overline{\delta}=\pi\circ \delta$ and trivialize it via the section $\delta$. After identifying the total space with $\T^2=\S^1\times \S^1$, where the bundle projection $\pi$ corresponds to the projection to the first circle, the loop  $\delta\ast \gamma_V^k$ reads $(\S^1\times \{0\} )\ast (\{0\}\times \S^1)^k$ and the loop  $\phi^X_{f\circ\pi}\circ \delta$ reads 
$t\mapsto (t,f(t))$, where $f:\S^1\to\S^1$ of degree $k$. It is clear that these loops are homotopic in $\T^2$.     
\end{proof}
As a consequence of the lemma we get the following class of prequantizations where the map \eqref{HH to Cont} is injective due to topological reasons, more precisely due to the fact that \eqref{HH to Diff} is injective.
\begin{proposition}\label{baby prop}
Let $(M,\omega)$ be a closed integral symplectic manifold such that
\begin{enumerate}[label=\text{(\arabic*)}]
    \item $[\omega]$ vanishes on $\pi_2(M)$.
    \item Any abelian subgroup of $\pi_1(M)$ has rank at most one.
\end{enumerate}
Let $(V,\lambda)\to (M,\omega)$ be a prequantization. Then for any closed integral 1-form $\alpha$ on $M$, $L_\alpha$ is isotopic to $\Id$ in $\Diff(V)$ if and only if $[\alpha]$ is trivial in $H_{dR}^1(M,\Z)$. 
\end{proposition}
\begin{lemma}\label{central extension}
    Let $H$ be an abelian group of rank at most one. Then every $\Z-$central extension of $H$ is trivial. More precisely, if $G$ is a group with a central subgroup $C$ isomorphic to $\Z$ and $G/C$ is isomorphic to $H$ then $G$ is isomorphic to $H\times \Z$. In particular, $G$ is abelian.
\end{lemma}
\begin{proof}
    The commutator subgroup $[G,G]$ is contained in $C=\Z$. The map $G\times G\to \Z$ that sends a couple $(a,b)$ to the commutator $aba^{-1}b^{-1}$ factors to a skew-symmetric bilinear map $B: H\times H\to \Z$. Since $H$ has rank at most one, it follows that $B$ is identically zero. This proves the Lemma.
\end{proof}
\begin{proof}[Proof of Proposition \ref{baby prop}]
The homotopy exact sequence associated to the bundle $V\to M$ reads 
\begin{equation}\label{hom ex seq}
 ...\to\pi_2(V)\to \pi_2(M)\to\pi_1(\S^1)\to \pi_1(V)\to \pi_1(M)\to 0   
\end{equation}
where the connecting homomorphism is given by integrating $\omega$ over spheres in $M$. By the first assumption, the fiber $\gamma_V$ is of infinite order in $\pi_1(V)$. We note also that the image of $\pi_1(\S^1)$ is always central in $\pi_1(V)$, that is $\gamma_V$ commutes with every element of $\pi_1(V)$. In fact, given any loop $\eta$ in $V$ (having the same base point with $\gamma_V$) and its projection $\overline{\eta}$ in $M$, $\gamma_V$ and $\eta$ can be seen as loops in the total space of the pull-back bundle $\overline{\eta}^*V\to \S^1$, which is diffeomorphic to $\T^2$ and therefore has abelian fundamental group. 

Now if $L_\alpha$ is isotopic to $\Id$ and  $\delta$ is a loop in $V$, then by Lemma \ref{L_alpha on pi_1}, $\delta$ is freely homotopic to $\delta\ast \gamma_V^k$. Then there is a loop $\beta$ in $V$ such that $\beta\ast\delta\ast\beta^{-1}=\delta\ast \gamma_V^k$ in $\pi_1(V)$. We write $\overline{\delta}=\pi\circ \delta$ and $\overline{\beta}=\pi\circ \beta$. Then $\overline{\delta}$ commutes with $\overline{\beta}$ in $\pi_1(M)$ and by assumption $\overline{\delta}$ and $\overline{\beta}$ are contained in a subgroup $H$ of rank at most one. Then, by Lemma \ref{central extension}, the subgroup $\pi_*^{-1}(H)$ is abelian. As $\delta$ and $\beta$ are contained in $\pi_*^{-1}(H)$, we conclude that $\gamma_V^k$ is contractible and therefore $k=\int_{\overline{\delta}} \alpha=0$. Since $\delta$ is arbitrary we get $[\alpha]=0$. 
\end{proof}
\begin{remark}\label{flux group vanishes} The above statement shows that the flux group $\Gamma_\omega$ is trivial in this case and $\pi_0(\HH)$, being isomorphic to $H_{dR}^1(M,\Z)$, injects into $\pi_0(\Diff(V))$.
\end{remark}

\begin{example} A typical example for the above proposition is an oriented surface of genus at least two.
\end{example}
Our next objective is to use Proposition \ref{baby prop} as a springboard to describe prequantizations for which the homomorphism \eqref{HH to Cont} is injective for contact topological reasons. We want to replace the role of the Reeb orbit with the loop induced by the Reeb flow in the bundle   $\CSp(\xi)\to V$  of conformally symplectic frames of $\xi$. 

The bundle $\CSp(\xi)\to V$ is a principal $\CSp(2n,\R)-$bundle, where $\CSp(2n,\R)$ denotes the group of conformally symplectic matrices. As the symplectic group $\Sp(2n,\R)$ is a deformation retract of $\CSp(2n,\R)$, it is natural to consider the isomorphism, called the \emph{Maslov index}, given by 
$$\Index: \pi_1(\CSp(2n,\R))=\pi_1(\Sp(2n,\R))\to\pi_1(\U(n))\to\pi_1(\S^1)\cong \Z$$
where the first arrow is given by the polar decomposition and the second arrow is induced by the determinant map. We have the homotopy exact sequence
\begin{equation}\label{hom ex seq frame}
...\to \pi_2(V)\to \pi_1(\CSp(2n,\R))\to\pi_1(\CSp(\xi))\to \pi_1(V)\to 0    
\end{equation}
and it is easy to see that the connecting homomorphism is given by integrating the first Chern class $c_1(\xi)$ over spheres in $V$. Here we view $\xi$ as a complex vector bundle of rank $n$, by choosing a complex structure on $\xi$ compatible with $d\lambda$. We note that the space of such complex structures is non-empty and contractible. In the same vein $c_1(TM)$ is well-defined and we have $c_1(\xi)=\pi^*c_1(TM)$. Combining this with \eqref{hom ex seq frame}, we conclude that the homomorphism 
\begin{equation}\label{pi_1 to p_1}
\pi_1(\CSp(2n,\R))\to\pi_1(\CSp(\xi))    
\end{equation}
is injective if and only if $c_1(TM)$ vanishes on every sphere of vanishing symplectic area in $M$. Note that this is equivalent to assuming that $(M,\omega)$ is \emph{monotone on} $\pi_2(M)$, that is, there is some $\kappa\in\R$ (in fact $\kappa\in\Q$), which we call the \emph{monotonicity constant}, such that $c_1(TM)=\kappa\, \omega$ on $\pi_2(M)$. We also remark that the image of $\pi_1(\CSp(2n,\R))$ in $\pi_1(\CSp(\xi))$ is always central. This follows from the argument given in the proof of Proposition \ref{baby prop} for the bundle $V\to M$ and the fact that $\pi_1(\CSp(2n,\R))$ is abelian. 

Our aim now is to study the loop  in $\CSp(\xi)$ that is associated to the Reeb flow and to determine sufficient conditions such that it is of infinite order in $\pi_1(\CSp(\xi))$. Let $p\in V$ and $\mathcal{F}_p$ be a symplectic frame of $\xi_p$. We define the \emph{Reeb frame loop} in the frame bundle by
\begin{equation}\label{reeb frame loop}
\Gamma_V:\S^1\to \CSp(\xi),\ \ t\mapsto (\phi^X_t)_*\mathcal{F}_p, \end{equation}
which defines an element in $\pi_1(\CSp(\xi),\mathcal{F}_p)$. Next we assume that the Reeb orbit $\gamma_V$ is contractible. Let $j:\D\to V$ be a capping disc, that is $j=\gamma_V$ on $\partial \D$. Let $\mathcal{F}$ be a symplectic framing of $j^*\xi$.  Writing $\Gamma_V$ with respect to the framing $\mathcal{F}|_{\partial \D}$, we get the loop  $$\Gamma_V^j:\S^1\to \CSp(2n,\R)$$ 
to which we can associate its Maslov index.
\begin{lemma}\label{index formula} Let $\overline{j}:S^2:=\D/\partial \D\to M$ be the sphere given by $\overline{j}=\pi\circ j$. Then we have
$$\Index(\Gamma_V^j)=-\langle c_1(\overline{j}^*TM),S^2\rangle=-\langle c_1(TM),\overline{j}(S^2)\rangle.$$    
In particular, $\Index(\Gamma_V^j)$ does not depend on $\mathcal{F}$.
Moreover, $\Index(\Gamma_V^j)$ is independent of the capping disc if and only if $(M,\omega)$ is monotone on $\pi_2(M)$. 
\end{lemma}
\begin{proof} Let $x\in S^2$ denote the image of $\partial \D$ in the quotient. We take a collar neighbourhood $U$ of $\partial \D$ in $\D$ and let $\overline{U}$ be the corresponding disc neighbourhood of $x\in S^2$. We pick a symplectic framing $\overline{\mathcal{G}}$ of $\overline{j}^*TM|_{\overline{U}}$, say extending $\pi_*\mathcal{F}_p$, and lift it to a Reeb invariant framing $\mathcal{G}$ of $j^*\xi|_{U}$ extending $\mathcal{F}_p$. Note that this is possible since $\overline{U}$ is contractible. On the other hand  $\mathcal{F}$ descends to a symplectic framing $\overline{\mathcal{F}}$ of $\overline{j}^*TM|_{S^2\setminus \{x\}}$. Viewing $c_1(\overline{j}^*TM)$ as the homotopy class of the clutching function that is used to glue the trivial bundles $(\overline{j}^*TM|_{S^2\setminus \overline{U}}, \overline{\mathcal{F}})$ and $(\overline{j}^*TM|_{\overline{U}}, \overline{\mathcal{G}})$ along $\partial\overline{U}$, we get
$$\langle c_1(\overline{j}^*TM),S^2\rangle= \Index(\overline{\mathcal{F}}|_{\partial\overline{U}})=-\Index(\overline{\mathcal{G}}|_{\partial\overline{U}})$$
where we view $\overline{\mathcal{F}}|_{\partial\overline{U}}$, respectively $\overline{\mathcal{G}}|_{\partial\overline{U}}$,
as a loop in $\CSp(2n,\R)$ using the framing $\overline{\mathcal{G}}|_{\partial\overline{U}}$, respectively $\overline{\mathcal{F}}|_{\partial\overline{U}}$. Using the symplectic isomorphism $\pi_*$, we get 
$$\langle c_1(\overline{j}^*TM),S^2\rangle=-\Index({\mathcal{G}}|_{\eta})$$
where $\eta$ denotes the preimage of $\partial \overline{U}$ in $\D$ and the loop ${\mathcal{G}}|_{\eta}$ is obtained via the framing $\mathcal{F}$. Since ${\mathcal{G}}|_{\eta}$ is homotopic to the loop ${\mathcal{G}}|_{\partial \D}$, which is defined via the framing $\mathcal{F}$, and since $\Gamma^j_V$ reads as the constant loop with respect to the framing $\mathcal{G}|_{\partial \D}$, we get
$$\langle c_1(\overline{j}^*TM),S^2\rangle=-\Index({\mathcal{G}}|_{\partial \D})=-\Index(\Gamma_V^j),$$
where the left hand side is clearly independent of the framing $\mathcal{F}$. 

Recall that $(M,\omega)$ is monotone on $\pi_2(M)$ if and only if $c_1(\xi)$ vanishes on $\pi_2(V)$.
 Now given two capping discs $j_1,j_2:\D\to V$, we glue them together, say after reversing the orientation of the second, to get a sphere $j:S^2\to V$. We have 
$$\Index(\Gamma_V^{j_1})-\Index(\Gamma_V^{j_2})=\langle c_1(j^*\xi),S^2\rangle.$$
Hence $\Index(\Gamma_V^{j_1})=\Index(\Gamma_V^{j_2})$ if $c_1(\xi)$ vanishes on $\pi_2(V)$. On the other hand given a sphere $S$ in $V$, one can choose some capping disc $j_1:\D\to V$ with $j_1(0)\in 
 S$ and define the capping disc $j_2$ to be the concatenation of $j_1$ and $S$. Repeating the above construction for these two capping discs leads to a sphere $j$ in the homotopy class of $S$. Then the above formula says that $c_1(\xi)$ vanishes on $S$ if the index is independent of the capping discs.      

\end{proof}
\begin{remark} The above construction of the Maslov index can be carried out also in the case where $\gamma_V$ is null homologous. In this case the capping disc is replaced by a capping surface $j:\Sigma\to V$ and the relative index is given by evaluating $c_1(TM)$ over the closed surface $\overline{j}:\overline{\Sigma}:=\Sigma/\partial\Sigma\to M$. 
\end{remark}
An immediate consequence of the above lemma is the following. 
\begin{proposition}\label{reeb loop frame inf order} Let $(V,\lambda)\to (M,\omega)$ be a prequantization. Suppose that either $[\omega]$ vanishes on $\pi_2(M)$ or $(M,\omega)$ is monotone on $\pi_2(M)$ with non-vanishing monotonicity constant. Then the Reeb frame loop $\Gamma_V$ is of infinite order in $\pi_1(\CSp(\xi))$.  
\end{proposition}
\begin{proof} Notice that if $[\omega]$ is trivial on $\pi_2(M)$ then the Reeb orbit is of infinite order in $\pi_1(V)$. Since $\Gamma_V$ projects to $\gamma_V$ the statement follows.  

Now assume that there is sphere in $M$ with unit symplectic area. Then $\gamma_V$ is contractible.
Given a capping disc $j:\D\to V$ with $p=j(1)$, with a framing $\mathcal{F}$ of $j^*\xi$ extending $\mathcal{F}_p$, the loop $\Gamma_V$, based at $\mathcal{F}_p$ is homotopic to a loop in $\CSp(\xi_p)$ based at $\mathcal{F}_p$, whose homotopy class is given by $\Index(\Gamma_V^j)$ after $(\CSp(\xi_p),\mathcal{F}_p)$ is identified by $(\CSp(2n,\R),\Id)$. By monotonicity, the homomorphism \eqref{pi_1 to p_1} is injective. Hence $\Gamma_V$ is of infinite order in $\pi_1(\CSp(\xi),\mathcal{F}_p)$ if and only if $\Index(\Gamma_V^j)$ is non-zero. 

On the other hand,
$$1=\int_{\gamma_V}\lambda=\int_{\D}j^*d\lambda=\int_{S^2}\overline{j}^*\omega.$$
Since the monotonicity constant is assumed to be non-zero, $\Index(\Gamma_V^j)$ does not vanish.

Finally, if $\gamma_V$ is torsion, say of order $d$, in $\pi_1(V)$, then we consider the loop $\Gamma_V^d$ in $\pi_1(\CSp(\xi))$ and apply the Maslov index construction to $\Gamma_V^d$. The argument above applies word by word to $\Gamma_V^d$ and we conclude that $\Gamma_V^d$ and consequently $\Gamma_V$ is of infinite order.    
\end{proof}    
Given $p\in V$ and a symplectic frame $\mathcal{F}_p$ of $\xi_p$, we define a map
$$\Cont(V,\xi)\to \CSp(\xi),\ \ \varphi\mapsto \varphi_*\mathcal{F}_p$$
and get the induced homomorphisms $\pi_1(\Cont(V,\xi),\Id)\to \pi_1(\CSp(\xi),\mathcal{F}_p)$ which maps the homotopy class of the Reeb loop $(\phi^X_t)_{t\in\S^1}$ to the homotopy class of $\Gamma_V$. Hence the Reeb loop is of infinite order in $\pi_1(\Cont(V,\xi),\Id)$ if $\Gamma_V$ is of infinite order in $\pi_1(\CSp(\xi),\mathcal{F}_p)$. Hence we get the following corollary.
\begin{corollary} \label{reeb loop cont inf order} Let $(V,\lambda)\to (M,\omega)$ be a prequantization. Suppose that either $[\omega]$ vanishes on $\pi_2(M)$ or $(M,\omega)$ is monotone on $\pi_2(M)$ with non-vanishing monotonicity constant. Then the Reeb loop $(\phi^X_t)_{t\in\S^1}$ is of infinite order in $\pi_1(\Cont(V,\xi))$.    
\end{corollary}  

After establishing the setting in which $\Gamma_V$ is suitable to take the role of $\gamma_V$ in the proof of Theorem \ref{baby prop}, we are ready to prove the lemma that is analogous to Lemma \ref{L_alpha on pi_1}. 
\begin{lemma}\label{L_alpha on pi_1 frame}  Let $\widehat{\delta}:\S^1\to \CSp(\xi)$ be a loop with its projection $\delta:\S^1\to V$. Then $(L_\alpha)_*\circ \widehat{\delta}$ is freely homotopic to the loop $\widehat{\delta}\ast \Gamma_V^k$ where 
$\Gamma_V:\S^1\to \CSp(\xi)$ is the Reeb frame loop based at $\widehat{\delta}(0)$ and $k=\int_\delta \pi^*\alpha$.    
\end{lemma}
\begin{proof} Recall that $L_\alpha=\varphi_1\circ \phi^X_{f\circ \pi}$ where, $\alpha=df$ for some $f:M\to\S^1$ and $\varphi_1$ is the time-one-map of the flow $(\varphi_t)$ of the horizontal vector field $X_h$. From \eqref{pullback by hor. flow} it follows that 
$$(\varphi_t)_*:(\xi_0,d\lambda)\to (\xi_t,d\lambda),\ \ t\in [0,1]$$
is a smooth path of isomorphisms of rotating contact structures
$$\xi_t:=\ker (\lambda+(t-1)\pi^*\alpha),\ \ t\in [0,1],$$
seen as symplectic vector bundles. Note that $\xi_1=\ker \lambda= \xi$ and by \eqref{GId on P} we have that 
$$(\phi^X_{f\circ \pi})_*:(\xi_1,d\lambda)\to (\xi_0,d\lambda)$$
is also an isomorphism. Notice that the vector field $X$ is the Reeb vector field for every contact form $\lambda+(t-1)\ \pi^*\alpha$ and by projecting along $X$, we get a smooth path
$$P_t:(\xi_t,d\lambda)\to (\xi_1,d\lambda)$$
of isomorphisms of symplectic vector bundles. Observe that $P_1$ is the identity map. Now  we define a path of loops 
$$\sigma: \S^1\times [0,1]\to \CSp(\xi),\ \ (s,t)\mapsto P_t\circ (\varphi_t)_*\circ (\phi^X_{f\circ \pi})_*\circ \widehat{\delta}(s),$$
which defines a free homotopy between $\sigma(\cdot,0)=P_0\circ (\phi^X_{f\circ \pi})_*\circ\widehat{\delta}$ and $\sigma(\cdot,1)=(L_\alpha)_*\circ \widehat{\delta}$. 

We claim that $P_0\circ (\phi^X_{f\circ \pi})_*\circ\widehat{\delta}$ is freely homotopic to $\widehat{\delta}\ast \Gamma_V^k$ in $\CSp(\xi)$. To see this, we write $\overline{\delta}:=\pi\circ \delta$ and consider the pull-back bundle $\overline{\delta}^*V\to \S^1$. Let $\Delta: \overline{\delta}^*V\to V$ be the bundle map covering $\overline{\delta}$. Next we take a symplectic trivialization of $\overline{\delta}^*TM$ and lift it to an $\S^1-$invariant trivialization of the pull-back bundle $\Delta^*\xi\to \overline{\delta}^*V$. Pairing this trivialization with the trivialization of $\overline{\delta}^*V\to \S^1$ described in the proof of Lemma \ref{L_alpha on pi_1}, we identify $\CSp(\Delta^*\xi)$ with $\T^2\times \CSp(2n,\R)$ in such a way that $\widehat{\delta}$ reads as $(\S^1\times \{0\})$ in $\T^2$ paired with a loop $\eta$ in $\CSp(2n,\R)$ based at $\Id$ and $\Gamma_V^k$ reads as $(\{0\}\times \S^1)^k$ in $\T^2$ paired with the constant loop $\Id$ in $\CSp(2n,\R)$ since $\Delta^*\xi$ is trivialized in an $\S^1-$invariant fashion. Hence $\widehat{\delta}\ast \Gamma_V^k$ reads as $(\S^1\times \{0\})\ast (\{0\}\times \S^1)^k$ paired with $\eta$. On the other hand since $P_0$ is given by the projection along $X$, for any $s\in \S^1$ and $u\in \widehat{\delta}(s)$, we have
\begin{equation*}
P_0\circ (\phi^X_{f\circ \pi})_*u=P_0\left(d(\phi^X_{f(\overline{\delta}(s))})_{\delta(s)}\ [u]+\alpha_{\overline{\delta}(s)}(\pi_*u)X\right)=d(\phi^X_{f(\overline{\delta}(s))})_{\delta(s)}\ [u].    
\end{equation*}
Again since  $\Delta^*\xi$ is trivialized in an $\S^1-$invariant fashion, $P_0\circ (\phi^X_{f\circ \pi})_*\circ\widehat{\delta}$ reads
$$s\mapsto\left((s,f(s)),\eta(s)\right)\in \T^2\times \CSp(2n,\R),$$
where $f:\S^1\to \S^1$ is of degree $k$.
Hence $P_0\circ (\phi^X_{f\circ \pi})_*\circ\widehat{\delta}$ is freely homotopic to $\widehat{\delta}\ast \Gamma_V^k$ in $\T^2\times \CSp(2n,\R).$ 
\end{proof}
Now we are ready to state our first main result. 
\begin{theorem}\label{main thm}
Let $(M,\omega)$ be a closed integral symplectic manifold such that either $[\omega]$ vanishes on $\pi_2(M)$ or $(M,\omega)$ is monotone on $\pi_2(M)$ with a non-vanishing monotonicity constant. Suppose that any abelian subgroup of $\pi_1(M)$ has rank at most one. If $(V,\lambda)\to (M,\omega)$ is a prequantization, then for any closed integral 1-form $\alpha$ on $M$, $L_\alpha$ is isotopic to $\Id$ in $\Cont(V,\xi)$ if and only if $[\alpha]$ is trivial in $H_{dR}^1(M,\Z)$.
\end{theorem}
\begin{proof} If $[\omega]$ vanishes on $\pi_2(M)$ then this follows from Proposition \ref{baby prop}. Now, suppose $[\omega]$ does not vanish on $\pi_2(M)$ and $(M,\omega)$ is monotone on $\pi_2(M)$ with a non-vanishing monotonicity constant. If  $L_\alpha$ is isotopic to $\Id$ in $\Cont(V,\xi)$ and  $\widehat{\delta}$ is a loop in $\CSp(\xi)$, then by Lemma \ref{L_alpha on pi_1 frame}, $\widehat{\delta}$ is freely homotopic to $\widehat{\delta}\ast \Gamma_V^k$. So there is a loop $\widehat{\beta}$ in $\CSp(\xi)$ such that $\widehat{\beta}\ast\widehat{\delta}\ast\widehat{\beta}^{-1}=\widehat{\delta}\ast \Gamma_V^k$ in $\pi_1(\CSp(\xi))$. Let  $\delta$ and $\beta$ be the projections of $\widehat{\delta}$ and $\widehat{\beta}$ to $V$ respectively. Then $\delta^{-1}\beta\delta\beta^{-1}=\gamma_V^{k}$ where $\gamma_V^{k}$ is the projection of $\Gamma_V^k$. But, we can assume $\gamma_V^{k}$ to be contractible by iterating $L_{\alpha}$ (or equivalently, by considering a suitable multiple $n\alpha$ of $\alpha$). In this case,  $\delta$ commutes with $\beta$ in $\pi_1(V)$. It follows $\delta$ and $\beta$ are contained in an abelian subgroup $H$ of rank at most one since the $\S^1-$fiber of $V$ is torsion in $\pi_1(V)$ (by the assumption that $[\omega]$ does not vanish on $\pi_2(M)$) and every abelian subgroup of $\pi_1(M)$ has rank at most one. It follows, by Lemma \ref{central extension}, that the subgroup $\pi_*^{-1}(H)$ is an abelian subgroup of $\pi_1(\CSp(\xi))$ since it is a $\Z-$central extension of $H$ (the Reeb frame loop is of infinite order by Proposition \ref{reeb loop frame inf order}). Hence $\widehat{\delta}$ and $\widehat{\beta}$ are contained in the abelian subgroup $\pi_*^{-1}(H)$ and therefore  $\Gamma_V^k$ is contractible. But  by Proposition \ref{reeb loop frame inf order} this is possible only if  $k=\int_{\overline{\delta}} \alpha=0$. Since $\delta$ is arbitrary we get $\alpha=0$ in $H_{dR}^1(M,\Z)$.     
\end{proof}

Now let $(W,\omega_W)$ be a simply connected integral symplectic manifold. Consider the product symplectic manifold $M:=W\times \Sigma_g$, $\omega:=\omega_W\oplus \omega_0$, where $\omega_0$ is any integral symplectic form on the surface $\Sigma_g$ of genus $g\geq 2$ and let $\pi:(V,\lambda)\to (M,\omega) $ be a prequantization. Then $V$ restricted to $W\times \{pt\}$, denoted by $V|_W$, defines a prequantization of $(W,\omega_W)$. Note that $\pi_1(M)=\pi_1(\Sigma_g)$ and $\Gamma_\omega$ is trivial (see, for example, \cite{CamPed} Theorem 3) . Then by Theorem \ref{comp. of H}, the kernel of the homomorphism (which is surjective in this case) 
$$\pi_0(\SCont(V,\lambda))\to\pi_0(\Symp(M,\omega))$$ 
is given by $H^1_{dR}(M,\Z)=H^1_{dR}(\Sigma_g,\Z)\cong \Z^{2g}$.
So we get the homomorphism  $$\mathcal{I}:H^1_{dR}(\Sigma_g,\Z)\to \pi_0(\Cont(V,\xi))$$ and we have the natural homomorphism $\mathcal{J}: \pi_0(\Cont(V,\xi))\to \pi_0(\Diff(V)).$
\begin{proposition} \label{smooth} Suppose that the $\S^1-$action on $V|_W$, given by the Reeb flow, defines a torsion loop in $\pi_1(\Diff_0(V|_W))$ of order $k$. Then every mapping class in the subgroup 
$kH^1_{dR}(\Sigma_g,\Z)\subset \pi_0(\SCont(V,\lambda))$ is smoothly trivial.
\end{proposition}
\begin{proof} Let $\sigma:\S^1\times [0,\epsilon]\to \Sigma_g$ be a symplectic embedding of a small cylinder and let $f:[0,\epsilon]\to \S^1=\R/\Z$ be a smooth function that vanishes near $0$ and $\epsilon$. Now, $f$ induces a function $\widetilde{f}:\Sigma_g\to \S^1$ as follows: extend $f$ to $\S^1\times [0,\epsilon]$ in the obvious way, so it induces a function on the image of $\sigma$ which can be extended to be identically zero outside. The differential form $d\widetilde{f}$ represents an integral cohomology class and, conversely, every integral cohomology is represented by such a function (with a suitable choice of $\sigma$ and the degree of $f$). Therefore, every element of $H^1_{dR}(M,\Z)$ can be represented as the differential of a function $\widetilde{f}\circ p$ where $p:M\to \Sigma_g$ is the projection and $\widetilde{f}$ is as above.

Recall that by \eqref{L_alpha}, the lift of $d(\widetilde{f}\circ p)$ to $\SCont(V,\lambda)$ is given by $\phi^X_{(\widetilde{f}\circ p)\circ\pi}\circ \varphi_1$  
 where $\varphi_1$ is the time-one-map of the flow $(\varphi_t)$ of the horizontal lift of the symplectic gradient of $d(\widetilde{f}\circ p)$. Since $\varphi_1$ is isotopic to identity, then one needs only to show that if the degree of $f$ is $k$ (where $f:[0,\epsilon]\to \S^1$ is seen as a loop) then $\phi^X_{(\widetilde{f}\circ p)\circ\pi}$ is isotopic to identity.
 
Let $\sigma:\S^1\times [0,\epsilon]\to \Sigma_g$ be as above and $f:[0,\epsilon]\to \S^1$ be such a map of degree $k$. Put $C$ to be the image of $\sigma$. Let $V|_C$ denote the restriction of the $\S^1-$bundle $V$ to $W\times C$ and $Z=V|_W$ denote the restriction of $V$ to $W\times \{pt\}$. Then the map $p\circ \pi:V|_C\to C$ defines a fiber bundle with fiber $Z$ and the $\S^1-$action preserves each fiber. Our claim is that $p\circ \pi:V|_C\to C$ is equivalent, in an $\S^1-$equivariant way, to $Z\times C\to C$ where $Z$ is seen as the prequantization of $(W,\omega_W)$. 
 
We first show that the bundle $p\circ \pi:V|_C\to C$ is trivial, in an $\S^1-$equivariant way, when restricted to $\S^1\times\{0\}\subset C$, where $C$ is identified with $\S^1\times [0,\varepsilon]$.
Fix $y_0\in W$ and suppose, without loss of generality, that the loop $S_{y_0}:=\{y_0\}\times \S^1 \times \{0\}$ has zero holonomy. Then for every $y\in W$, the loop $S_{y}:=\{y\}\times \S^1 \times \{0\}$ has also zero holonomy since for some $\varphi\in \Ham(W,\omega_W)$ with $\varphi(y_0)=y$, the Hamiltonian diffeomorphism $\varphi\times \Id$ of $M$ maps $S_{y_0}$ to $S_{y}$. Now, horizontal lifts of the loops $S_y$ for $y\in W$ produce an $\S^1-$equivariant identification of all the fibers $(p\circ \pi)^{-1}(z)$ for $z\in \S^1\times \{0\}$. 
Now we can argue in the same fashion and identify all the fibers $(p\circ \pi)^{-1}(z)$, for $z\in \S^1\times [0,\varepsilon]$, $\S^1-$equivariantly just by considering the foliation of $\S^1\times [0,\varepsilon]$ by intervals $\{s\}\times [0,\varepsilon]$ and their associated horizontal lifts.

Now, the map $\phi^X_{(\widetilde{f}\circ p)\circ\pi}$ reads $$\phi^X_{(\widetilde{f}\circ p)\circ\pi}: Z\times C\to  Z\times C,\ \ (q,(x,y))\mapsto (f(y)\cdot q,(x,y))$$
 where $f(y)\cdot q$ means the action of $f(y)\in \S^1$ on $q$. So, $\phi^X_{(\widetilde{f}\circ p)\circ\pi}$ has the form 
 $$(q,(x,y))\mapsto (\phi_y(q), (x,y))$$ where $\phi:[0,\epsilon]\to \Diff^0(Z)$ is a loop based at $\Id$. In our case, $\phi$ is contractible if $f$ has degree $k$ (since, by assumption, the Reeb flow defines a torsion loop of order $k$ in $\pi_1(\Diff^0(Z))$). This shows that $\phi^X_{(\widetilde{f}\circ p)\circ\pi}$ is isotopic to identity.
\end{proof}
\begin{corollary}\label{cor main} Let $(\Sigma_g,\omega_0)$ be a surface of genus $g\geq 2$ with an integral symplectic form and $\omega_{FS}$ be the Fubini-Study form on $\C P^n$, which is normalized so that the standard $\C P^1\subset \C P^n$ has unit area. Consider a prequantization  
$(V,\lambda)\to (\C P^n\times \Sigma_g, \omega_{FS}\oplus \omega_0).$ Then the homomorphism 
$$\mathcal{I}:H^1_{dR}(\Sigma_g,\Z)\to \pi_0(\Cont(V,\xi)),$$ 
is injective. Moreover $\ker(\J)$ contains the image of $\I$ for $n$ odd, and contains $\I(2H^1_{dR}(\Sigma_g,\Z))$ for $n$ even.
\end{corollary}
\begin{proof} The injectivity of $\I$ follows from Theorem \ref{main thm}. Note that the prequantization $V|_{\C P^n}$ is the Hopf bundle $\S^{2n+1}\to \C P^n$ and the Reeb loop is given by the loop $t\mapsto diag(e^{2\pi it},...,e^{2\pi it})$ in $\U(n+1)$. This defines a loop in $\SO(2n+2)$ which is contractible if $n$ is odd and of order $2$ if $n$ is even. The rest is a direct consequence of Proposition \ref{smooth}.
\end{proof}

\end{document}